\documentclass[12pt,amsthm,amscd,amssymb,verbatim,epsf,amsmath,amsfonts]{amsart}

\theoremstyle{definition}
\newtheorem{exmp}{Example}[section]
\usepackage[colorlinks=true,linkcolor=blue,citecolor=blue]{hyperref}
\usepackage{enumerate}

\def\line#1{\hbox to \hsize{#1\hfill}}

\usepackage{dirtytalk,amsmath,amsfonts,amssymb,amscd,mathrsfs,graphicx,pst-node,tikz-cd,fancyvrb,lipsum,amsthm,graphicx,epsf,amsmath,verbatim}
\usepackage[colorlinks=true,linkcolor=blue,citecolor=blue]{hyperref}
\usepackage{times}
\usepackage{enumerate}

\def\line#1{\hbox to \hsize{#1\hfill}}
\begin{document}
\theoremstyle{plain}
\newtheorem{Thm}{Theorem}
\newtheorem{Cor}{Corollary}
\newtheorem{Con}{Conjecture}
\newtheorem{Main}{Main Theorem}
\newtheorem{Lem}{Lemma}
\newtheorem{Prop}{Proposition}
\def\R{\mathbb{R}}
\def\F{\mathbb{F}}
\def\Re{{\frak R\frak e}}
\def\Im{{\frak I\frak m}}
\def\S{\mathbb{S}}
\def\H{\mathbb{H}}
\def\L{\mathbb{L}}
\theoremstyle{definition}
\newtheorem{Def}{Definition}
\newtheorem{Note}{Note}

\newtheorem{example}{\indent\sc Example}

\theoremstyle{remark}
\newtheorem{notation}{Notation}
\renewcommand{\thenotation}{}

\errorcontextlines=0
\numberwithin{equation}{section}
\renewcommand{\rm}{\normalshape}%

\title[Null hypersurfaces in 4-manifolds]%
   {Null hypersurfaces in 4-manifolds endowed with a product structure}
\author{Nikos Georgiou}
\address{}
\email{}

\date{20 December 2022}

\address{Nikos Georgiou\\
 Department of Computing and Mathematics \\
 South East Technological University (SETU) \\
Waterford \\
Ireland.}


\maketitle

\let\thefootnote\relax\footnote{}
\begin{abstract}
In a 4-manifold, the composition of a Riemannian Einstein metric with an almost paracomplex structure that is isometric and parallel, defines a neutral metric that is conformally flat and scalar flat. In this paper, we study hypersurfaces that are null with respect to this neutral metric and in particular we study their geometric properties with respect to the Einstein metric. Firstly, we show that all totally geodesic null hypersurfaces are scalar flat and their existence implies that the Einstein metric in the ambient manifold must be Ricci-flat. Then, we find a necessary condition for the existence of null hypersurface with equal non-trivial principal curvatures and finally, we give a necessary condition on the ambient scalar curvature, for the existence  of null (non-minimal) hypersurfaces that are of constant mean curvature. 
\end{abstract}
\section{Introduction}


\noindent Einstein Riemannian 4-manifolds $(M,g)$ with a parallel, isometric, almost paracomplex structure $P$ exhibit many interesting properties through the metric $g'$ defined by $g'=g(P.,.)$. In particular, the metric $g'$ is of neutral signature, locally conformally flat, scalar flat and shares the same Levi-Civita connection and Ricci tensor with $g$ \cite{GG}. 

\noindent Recently, F. Urbano in \cite{urbano} and later, D. Gao, H. Ma and Z. Yao in \cite{GMY}, have studied hypersurfaces in ${\mathbb S}^2\times {\mathbb S}^2$ and ${\mathbb H}^2\times {\mathbb H}^2$, respectively, endowed with the Einstein product metric. In particular, they used two complex structures $J_1,J_2$ on those manifolds to study isoparametric and homogeneous hypersurfaces by considering the product $P=J_1J_2$, which is an (almost) paracomplex structure that is parallel and isometric with respect to the product metric.

\noindent The space ${\mathbb L}(M^3)$ of oriented geodesics in the 3-dimensional non-flat real space form $M^3$ is a 4-dimensional manifold admiting an Einstein metric and a paracomplex structure $P$ that is isometric and parallel. Therefore, there exists a neutral, locally conformally flat and scalar flat metric sharing the same Levi-Civita connection and Ricci tensor with the Einstein metric (see \cite{AGK} and \cite{An4} for more details). The paracomplex structure $P$ has been explicitly described by H. Anciaux in \cite{An4} in a similar manner as in the product of surfaces. More precisely, H. Anciaux constructed two (para) complex structures $J_1$ and $J_2$, so that $J_1J_2=J_2J_1$ and then considered the product $P=J_1J_2$. This paracomplex structure was used in \cite{GG1}, to study a class of hypersurfaces in ${\mathbb L}(M^3)$, called \emph{tangential congruences}, that are sets of all tangent oriented geodesics in a given surface in $M$. Particularly, it was shown that tangential congruences are null with respect to the neutral metric and if, additionally, they are tangent to a convex surface then they admit a contact structure. The space ${\mathbb L}({\mathbb R}^3)$ of oriented lines in ${\mathbb R}^3$ is also a 4-dimensional manifold admiting a neutral metric $G$ that is locally conformally flat, scalar flat and is invariant under the Euclidean motions \cite{AGK, GK1}. M. Salvai showed that $G$ is the only metric that is invariant of the group action of the Eucliean 3-space. The null hypersurfaces in ${\mathbb L}({\mathbb R}^3)$ play an important role in the study of the ultrahyperbolic equation 
\begin{equation}\label{e:ultrahyperbolic}
u_{x_1x_1}+u_{x_2x_2}-u_{x_3x_3}-u_{x_4x_4}=0,
\end{equation}
where, $u=u(x_1,x_2,x_3,x_4)$ is a real function in ${\mathbb R}^4$ (see \cite{CG}). Specifically, let ${\mathbb R}^{2,2}=({\mathbb R}^4, g_0:=dx_1^2+dx_2^2-dx_3^2-dx_4^2)$, and $f:{\mathbb L}({\mathbb R}^3)\rightarrow {\mathbb R}^{2,2}$ be the conformal map defined according to $G=\omega^2f^{\ast}g_0$, where $\omega$ is a strictly positive function. A function $v$ is harmonic with respect to $G$, i.e., $\Delta_{G}u=0$, if and only if $\omega\cdot v\circ f$ is a solution of the ultrahyperbolic equation (\ref{e:ultrahyperbolic}) \cite{CG}. This implies solving the ultrahyperbolic equation is equivalent to solving the Laplace equation with respect to the neutral metric $G$. Consider now the problem
\[
\Delta_G v=0,
\]
where the function $v$ on ${\mathbb L}({\mathbb R}^3)$, is given on the null hypersurface $H=\{\gamma\in {\mathbb L}({\mathbb R}^3)|\,\, \gamma\parallel P_0\}$, with $P_0$ is a fixed plane in ${\mathbb R}^3$. In \cite{BG} B. Guilfoyle presented an inversion formula describing $v$ on ${\mathbb L}({\mathbb R}^3)$, using Fritz John's inversion formula (cf. \cite{J}). It is then natural to ask whether an arbitrary real function defined on a null hypersurface can be uniquely extended to a harmonic function  on ${\mathbb L}(M^3)$ with respect of the neutral metric, for any 3-dimensional real space form $M^3$.

\noindent In this article, we study null hypersurfaces with respect to the neutral metric $g_-$ of an Einstein 4-dimensional manifold $(M,g_+)$ endowed with an almost paracomplex structure $P$ that is parallel and isometric, so that $g_-=g(P_+.,.)$.

\vspace{0.1in}

\noindent Our first result deals with totally geodesic null hypersurfaces. In particular, we have the following:

\vspace{0.1in}

\noindent {\bf Theorem 1.} \emph{Every totally geodesic null hypersurface is scalar flat. If $M$ admits a totally geodesic null hypersurface then $(M,g_+)$ is Ricci-flat.}

\vspace{0.1in}

\noindent Let $N$ be the unit normal vector field, with respect to the Riemannian Einstein metric $g_+$, along a null hypersurface. The principal curvature corresponding to the principal direction $PN$, is zero. The other two principal curvatures are called \emph{non-trivial}. The next result provide a necessary condition for the existence of null hypersurfaces with equal non-trivial principal curvatures.

\vspace{0.1in}

\noindent {\bf Theorem 2.} \emph{Suppose $(M,g)$ has nonnegative scalar curvature and $\Sigma$ is a null hypersurface with equal non-trivial principal curvatures. Then, $g$ is Ricci-flat and $\Sigma$ is totally geodesic.}

\vspace{0.1in}

\noindent Finally, we study (non-minimal) null hypersurfaces having constant mean curvature (CMC). In particular, we prove the following:

\vspace{0.1in}

\noindent {\bf Theorem 3.} \emph{Let $\Sigma$ be a CMC, non-minimal null hypersurface in $(M,g)$. Then, all principal curvatures and the scalar curvature of $\Sigma$ are constant. Furthermore, the scalar curvature of $g$ is given by
\[
\bar R=-8\lambda_1\lambda_2,
\]
where $\lambda_1,\lambda_2$, denote the non-trivial principal curvatures of $\Sigma$.
}

\vspace{0.2in}

\noindent {\bf Acknowledgements.} The author would like to thank T. Lyons for his helpful and valuable suggestions and comments.

\vspace{0.2in}


\vspace{0.1in}

\section{Preliminaries} \label{SectionOne}

\noindent Let $(M,g)$ be an Einstein 4-manifold endowed with a product structure $P$ (specifically a type (1,1) tensor field with $P^2=\mbox{Id}$) such that: 
\begin{enumerate}
\item The eigenbundles corresponding to the eigenvalues $+1$ and $-1$, have equal rank.
\item $P$ is an isometry, that is, $$g(P.,P.)=g(.,.).$$
\item $P$ is parallel, that is, $$\overline\nabla P=0,$$
where, $\overline\nabla$ is the Levi-Civita connection of $g$. 
\end{enumerate} 
In other words, $P$ is an almost paracomplex structure that is parallel and isometric.

\noindent Define the metric $g_-$ by, $$g_-=g(P.,.),$$ and denote $g$ by $g_+$.
Then, $g_-$ is of neutral signature, locally conformally flat and scalar flat \cite{GG}. Also, both metrics $g_+$ and $g_-$ share the same Levi-Civita connection $\overline\nabla$ (see \cite{An4} for further details).

\noindent Let $\Sigma^3$ be an oriented hypersurface of $M$ and consider the normal bundles:  
\[
\mathcal{N}_\pm(\Sigma)=\{\xi\in TM\,|\, g_\pm(X,\xi)=0,\, \forall \xi\in T\Sigma\}.
\]
Let $N_\pm$ be the unit normal vector of $\Sigma$ with respect to $g_\pm$, so that 
\[
g_\pm(N_\pm,N_\pm)=\epsilon_\pm\in\{-1,0,1\},
\]
(note that $\epsilon_+=1$) and define the functions $C_{\pm}$ on $\Sigma$ according to
\[
C_+=g_+(PN_+,N_+)=g_-(N_+,N_+),
\]
and
\[
C_-=g_-(PN_-,N_-)=g_+(N_-,N_-).
\]
Consider the tangential vector field along $M$
\[
X_\pm=PN_\pm-\epsilon_\pm C_\pm N_\pm.
\]
For $\xi\in \mathcal{N}_\pm(M)$, we have
\begin{eqnarray}
g_\pm(\nabla C_\pm,\xi)&=&\nabla_\xi C_\pm\nonumber \\
&=&2\,g_\pm(\nabla_\xi N_\pm,X_\pm)\nonumber\\
&=&g_\pm(\xi,-2A_\pm X_\pm),\nonumber
\end{eqnarray}
showing that
\begin{equation}\label{e:shapeoperator}
\nabla C_\pm=-2A_\pm X_\pm,
\end{equation}
where $A_\pm$ denotes the shape operator of $\Sigma$ immersed in $(M,g_\pm)$.

\noindent Also,
\begin{equation}\label{e:shapeoperator1}
\nabla_\xi X_\pm=-P^\bot A_\pm\xi+\epsilon_\pm C_\pm A_\pm\xi,
\end{equation}
where $P^\bot $ stands for the orthogonal projection of $P$ on $\Sigma$.
Let $R_\pm, H_\pm$ and $\sigma_\pm$ be respectively the scalar curvature, the mean curvature and the second fundamental form of $\Sigma$ immersed in $(M,g_\pm)$.

\begin{Prop}
The Hessian of $C_\pm$ is:
\begin{equation}\label{e:hessian}
\nabla^2C_\pm(u,v)=-2(\nabla_u\sigma_\pm)(X_\pm,v)-2\epsilon_\pm C_\pm g_\pm(A_\pm u,A_\pm v)+2g_\pm (PA_\pm u,A_\pm v).
\end{equation}
\end{Prop}
\begin{proof}
In this proof we omit the subscript $\pm$, unless is necessary. 

\noindent Using (\ref{e:shapeoperator}) on the tangential vector fields $u,v$, we have
\begin{eqnarray}
\nabla^2C(u,v)
&=& g(\nabla_u(-2AX),v)\nonumber\\
&=& -2g(\nabla_uAX,v)\nonumber\\
&=& -2\nabla_u(g(AX,v))+2g(AX,\nabla_uv)\nonumber\\
&=& -2\nabla_u(g(X,Av))+2g(AX,\nabla_uv)\nonumber\\
&=& -2g(\nabla_uX,Av)-2g(X,\nabla_uAv)+2g(AX,\nabla_uv)\nonumber\\
&=& -2g(\epsilon CAu-P^TAu,Av)-2g(X,\nabla_uAv)+2g(AX,\nabla_uv)\nonumber\\
&=& -2\epsilon Cg(Au,Av)+2G(PAu,Av)-2g(X,\nabla_uAv)+2g(AX,\nabla_uv)\nonumber
\end{eqnarray}
Note that $\sigma(u,v)=g(Au,v)$ and for simplicity use $\nabla_u\sigma(X,v)$ to denote $(\nabla_u\sigma)(X,v)$. We now have
\begin{eqnarray}
\nabla_u\sigma(X,v)&=& u(\sigma(X,v))-\sigma(\nabla_uX,v)-\sigma(X,\nabla_uv)\nonumber\\
&=& u(G(X,Av))-g(\nabla_uX,Av)-g(AX,\nabla_uv)\nonumber\\
&=& g(\nabla_uX,Av)+g(X,\nabla_uAv)-g(\nabla_uX,Av)-g(AX,\nabla_uv)\nonumber\\
&=&g(X,\nabla_uAv)-g(AX,\nabla_uv),\nonumber
\end{eqnarray}
and therefore,
\begin{eqnarray}
\nabla^2C(u,v)
&=&-2\epsilon Cg(Au,Av)+2g(PAu,Av)-2\nabla_u\sigma(X,v).\nonumber
\end{eqnarray}
\end{proof}

\begin{Prop}
If $\Delta$ denotes the Laplacian of the metric $g_+$ induced on the hypersurface $\Sigma$, then
\[
\Delta C_+=-6\,g_+(X_+,\nabla H_+)-2C_+|\sigma_+|^2+2\,\emph{Tr}(P^TA_+^2),
\]
where $H_+$ denotes the mean curvature and $A_+$ is the shape operator.
\end{Prop}
\begin{proof}
In the proof we omit the subscript $+$ unless is necessary. 
The Codazzi-Mainardi equation for $\Sigma$ is
\[
g(R(u,v)z,N)=(\nabla_u\sigma)(v,z)-(\nabla_v\sigma)(u,z).
\]
Consider the orthonormal frame $(e_1,e_2,e_3)$ of $\Sigma$, where $Ae_i=\lambda_i e_i$. The fact that $g$ is Einstein gives,
\begin{eqnarray}
\sum_{i=1}^3\left((\nabla_{e_i}\sigma)(X,e_i)-(\nabla_X\sigma)(e_i,e_i)\right)&=&\sum_{i=1}^3g(R(e_i,X)e_i,N)\nonumber\\
&=&\sum_{i=1}^3g(R(e_i,X)e_i,N)+g(R(N,X)N,N)\nonumber\\
&=&\overline{\mbox{Ric}}(X,N)\nonumber\\
&=&\textstyle{\frac{\bar R}{4}}\,g(X,N)\nonumber\\
&=&0.\nonumber
\end{eqnarray}
Thus, 
\begin{eqnarray}
\sum_{i=1}^3(\nabla_{e_i}\sigma)(X,e_i)&=&\sum_{i=1}^3(\nabla_X\sigma)(e_i,e_i)\nonumber\\
&=&\sum_{i=1}^3\nabla_X(\sigma(e_i,e_i))-\sigma(\nabla_Xe_i,e_i)-\sigma(e_i,\nabla_Xe_i)\nonumber\\
&=&3\nabla_XH-2\sum_{i=1}^3g(\nabla_Xe_i,Ae_i)\nonumber\\
&=&3g(X,\nabla H)-2\sum_{i=1}^3\lambda_i g(\nabla_Xe_i,e_i)\nonumber\\
&=&3g(X,\nabla H).\nonumber
\end{eqnarray}
Using the fact that $\epsilon_+=1$, we have
\begin{eqnarray}
\Delta C&=&\sum_{i=1}^3\nabla^2C(e_i,e_i)\nonumber\\
&=&-2\sum_{i=1}^3\left((\nabla_{e_i}\sigma)(X,e_i)+ Cg(Ae_i,Ae_i)-g(PAe_i,Ae_i)\right)  \nonumber\\
&=&-6g(X,\nabla H)-2\sum_{i=1}^3\left(\lambda^2_iC-\lambda^2_ig(Pe_i,e_i)\right), \nonumber
\end{eqnarray}
and this completes the proof.
\end{proof}

\noindent Let $R,R_{ij}, R_{ijkl}$ be respectively the scalar curvature, the Ricci tensor and the curvature tensor of the metric $g_+$ induced on $\Sigma$ and let $\bar R,\bar R_{ij}, \bar R_{ijkl}$ be respectively the scalar curvature, the Ricci tensor and the curvature of the ambient metric $g_+$.

\noindent Using the Gauss equation we get (for simplicity, we omit the subscript $+$):
\begin{eqnarray}
R&=& g^{ij}R_{ij}\nonumber \\
&=& g^{ij}g^{kl}(\bar R_{kilj}+\sigma_{ij}\sigma_{kl}-\sigma_{il}\sigma_{kj})\nonumber \\
&=& g^{ij}g^{kl}\bar R_{kilj}+9H^2-|\sigma|^2.\nonumber
\end{eqnarray}
The fact the $g_+$ is Einstein implies,
\begin{eqnarray}
g^{ij}g^{kl}\bar R_{kilj}&=&
g^{ij}\bar R_{ij} -g^{NN}\overline{\mbox{Ric}}(NN) \nonumber \\
&=&(\bar R_+-g^{NN}\overline{\mbox{Ric}}(NN))-g^{NN}\overline{\mbox{Ric}}(NN)\nonumber \\
&=&\bar R_+-2\overline{\mbox{Ric}}(NN)\nonumber \\
&=&\bar R_+-2(\bar R_+/4)g_+(N,N)\nonumber \\
&=&\bar R_+/2.\nonumber 
\end{eqnarray}
We then have
\begin{equation}\label{e:scalar curvature}
R_+=\textstyle{\frac{1}{2}}\bar R_++9H_+^2-|\sigma_+|^2.
\end{equation}
The Gauss equation for the metric $g_-$ induced on $\Sigma$, gives
\[
\bar R_-=R_-+2\overline{\mbox{Ric}}_-(N_-,N_-)+||\sigma_-||^2-9H_-^2,
\]
and using the fact that $g_-$ is scalar flat, we have that $\bar R_-=0$. 
Therefore,
\[
R_-=-2\overline{\mbox{Ric}}_-(N_-,N_-)-||\sigma_-||^2+9H_-^2.
\]
On the other hand,
\begin{eqnarray}
\overline{\mbox{Ric}}_-(N_-,N_-)&=&\overline{\mbox{Ric}}_+(N_-,N_-)\nonumber \\
&=&\textstyle{\frac{\bar R_+}{4}}C_-,\nonumber 
\end{eqnarray}
and thus,
\[
R_-=-\textstyle{\frac{\bar R_+}{2}}C_--||\sigma_-||^2+9H_-^2.
\]
We then have,
\begin{Prop}
Assume $(M,g_{\pm})$ has positive (resp. negative) scalar curvature $\bar R_{\pm}$. The following two statements hold:
\begin{enumerate}
\item If $\Sigma$ is a totally geodesic hypersurface in $(M,g_+)$, then it has positive (resp. negative) scalar curvature.
\item If $\Sigma$ is a totally geodesic hypersurface $(M,g_-)$, then it has negative (resp. positive) scalar curvature.
\end{enumerate}
\end{Prop}

\vspace{0.2in}


\section{Null Hypersurfaces}

\begin{Def}
A \emph{null hypersurface} in a pseudo-Riemannian manifold is an oriented hypersurface where the induced metric is indefinite and the normal vector field is null.
\end{Def}
\noindent In this section, when we refer to a null hypersurface we simply mean a hypersurface that is null with respect to the neutral metric of $g_-$. 
\begin{Prop}
Suppose $\Sigma$ is an oriented hypersurface of $M$. Then, the following statements hold:
\begin{enumerate}
\item $|C_+|\leq 1,\quad\mbox{and}\quad C_->0$.
\item $C_+=0$, if and only if $\Sigma$ is a null hypersurface.
\item If $\Sigma$ is a null hypersurface then, $PN_+$ is a principal direction with zero corresponding principal curvature.
\end{enumerate}
\end{Prop}
\begin{proof}
\begin{enumerate}
\item
It is not hard to confirm that $|X_+|=1-(C_+)^2\geq 0$. Also,
\[
C_-=g_+(N_-,N_-)>0.
\]
\item 
Assuming $C_+=0$, we have that $g_+(PN_+,N_+)=0$ and using the fact that $g_+$ is Riemannian then, $PN_+\in T\Sigma$. This implies, 
\[
g_-(PN_+,N_-)=0,
\]
or,
\[
g_+(N_+,N_-)=0.
\]
But this tells us that $N_-\in T\Sigma$, and therefore
\[
g_-(N_-,N_-)=0,
\]
which means that $\Sigma$ is null. 


\noindent Conversely, assume that $\Sigma$ is null and consider the non-zero normal vector field $N_-$. Then, $g_-(N_-,N_-)=0$. On the other hand, $g_-(N_-,T\Sigma)=0$, which means $g_+(PN_-,T\Sigma)=0$. Therefore, $PN_-=\lambda N_+$, where $\lambda\neq 0$, since $N_-$ is non-zero vector field. Thus,

\begin{eqnarray}
C_+&=& g_-(N_+,N_+)\nonumber\\
&=& \lambda^{-2}g_-(N_-,N_-)\nonumber\\
&=&0,\nonumber
\end{eqnarray}
and this completes the proof.

\item Since $\Sigma$ is null then $C_+=0$ and therefore
\[
X_+=PN_+-C_+N_+=PN_+\in T\Sigma.
\]
Note that 
\[
0=\nabla C_+=-2A_+X_+,
\]
which implies 
\[
A_+PN_+=0,
\]
and therefore $PN_+$ is a principal direction.
\end{enumerate}
\end{proof}

\vspace{0.2in}

\noindent For a null hypersurface $\Sigma$, we study the geometric properties of the metric $g_+$ induced on $\Sigma$ and for this reason we omit the $+$ subscripts unless it is necessary.

\vspace{0.2in}

\subsection{Examples of null hypersurfaces}

\begin{exmp}
We now describe the almost paracomplex structure defined in the spaces of oriented geodesics of 3-manifolds of constant curvature using their (para) K\"ahler structures (see \cite{AGK} \cite{GG1} \cite{GK1} \cite{salvai1} for more details). 

\noindent For $p\in\{0,1,2,3\}$, consider the (pseudo-) Euclidean 4-space ${\mathbb R}_p^4:=({\mathbb R}^4,\left<.,.\right>_p)$, where
\[
\left<.,.\right>_p=-\sum_{i=1}^p dX_i^2+\sum_{i=p+1}^4dX_i^2,
\]
and let ${\mathbb S}_p^{3}$ be the quadric
\[
{\mathbb S}_p^{3}=\{x\in {\mathbb R}^4|\; \left<x,x\right>_p=1\}.
\]
The quadric ${\mathbb S}_0^{3}$ is the 3-sphere ${\mathbb S}^{3}$, ${\mathbb S}_3^{3}\cap \{x\in {\mathbb R}^4|\, X_4>0\}$ is anti-isometric to the hyperbolic 3-space ${\mathbb H}^{3}$, ${\mathbb S}_1^{3}$ is the de Sitter 3-space $d {\mathbb S}^{3}$ and, ${\mathbb S}_2^{3}$ is anti-isometric to the anti-de Sitter 3-space $Ad {\mathbb S}^{3}$.

\noindent Let $g_p$ be the metric $\left<.,.\right>_p$ induced on ${\mathbb S}_p^{3}$ by the inclusion map. The space of oriented geodesics in ${\mathbb S}_p^{3}$ is a 4-dimensional manifold and is identified with the following Grasmmannian spaces of oriented planes on ${\mathbb R}_p^4$:
\[
{\mathbb L}^{\pm}({\mathbb S}_p^{3})=\{x\wedge y\in \Lambda^2({\mathbb R}_p^4)|\; y\in T_x{\mathbb S}_p^{3},\; g_p(y,y)=\pm 1\}.
\]

\noindent Let $\iota:{\mathbb L}^{\pm}({\mathbb S}_p^{3})\rightarrow \Lambda^2({\mathbb R}_p^4)$, be the inclusion map and $\left<\left<,\right>\right>_p$ be the flat metric in the 6-manifold $\Lambda^2({\mathbb R}_p^4)$, defined by
\[
\left<\left<u_1\wedge v_1,u_2\wedge v_2\right>\right>_p:=\left<u_1,u_2\right>_p\left<v_1,v_2\right>_p-\left<u_1,v_2\right>_p\left<u_2,v_1\right>_p.
\]
The metric $G_p=\iota^{\ast}\left<\left<,\right>\right>_p$ on  ${\mathbb L}^{\pm}({\mathbb S}_p^{3})$ is Einstein \cite{An4}. 

\noindent It was shown in \cite{GG}, that the Hodge star operator $\ast$ on the space of bivectors $\Lambda^2({\mathbb R}_p^4)$ in ${\mathbb R}_p^4$, restricted to the space of oriented geodesics ${\mathbb L}^{\pm}({\mathbb S}^3_p)$ defines an almost paracomplex structure ${\mathbb J}^{\ast}$ that is parallel and isometric with respect to the Einstein metric $G_p$. In particular, for $x\wedge y\in {\mathbb L}^{\pm}({\mathbb S}_p^{3})$, the almost paracomplex structure is defined by
\[
{\mathbb J}^{\ast}_{x\wedge y}=\left. \ast \right|_{T_{x\wedge y}{\mathbb L}^{\pm}({\mathbb S}^3_p)}.
\]
The metric $G'_p:=G_p({\mathbb J}^{\ast} .,.)$, is of neutral signature, locally conformally flat and scalar flat in ${\mathbb L}^{\pm}({\mathbb S}_p^{3})$.

\noindent Let $\phi:S\rightarrow{\mathbb S}_p^3$ be a non-totally geodesic smooth surface and $(e_1,e_2)$ be the principal directions of $\phi$ with corresponding eigenvalues $\kappa_1$ and $\kappa_2$. Then,
\[
\Phi:S\times {\mathbb S}^1\rightarrow{\mathbb L}({\mathbb S}_p^3):(x,\theta)\mapsto \phi(x)\wedge (\cos\theta\, e_1(x)+\sin\theta\, e_2(x)),
\]
is the immersion of the tangential congruence $\Sigma=\Phi(S\times {\mathbb S}^1)$ in the space of oriented geodesics ${\mathbb L}({\mathbb S}^3_p)$. It can be shown that if $\phi$ is a totally geodesic immersion, the mapping $\Phi$ is not an immersion. Also, $\Sigma$ is a null hypersurface with respect to the locally conformally flat neutral metric $g_-$ \cite{GG1}.

\noindent The eigenvalues of the tangential hypersurface $\Sigma$ are $0,\lambda_+$ and $\lambda_-$, where
\[
\lambda_+=\kappa_1\cos^2\theta+\kappa_2\sin^2\theta\qquad\lambda_-=-\kappa_1\sin^2\theta-\kappa_2\cos^2\theta,
\]
and therefore the mean curvature is
\[
H=\textstyle{\frac{1}{3}}(\kappa_1-\kappa_2)\cos2\theta.
\]
This yields:
\begin{Prop}
If $S$ is a totally umbilic surface in the non-flat 3-dimensional real space form, then the corresponding tangential congruence $\Sigma$ is a null hypersurface in $({\mathbb L}({\mathbb S}^3_p),G'_p)$ and is minimal in $({\mathbb L}({\mathbb S}^3_p),G_p)$. 
\end{Prop}
\end{exmp}

\vspace{0.1in}

\begin{exmp}
Consider the Cartesian product of the 2-spheres ${\mathbb S}^2\times {\mathbb S}^2$ endowed with the product metric 
\[
g_+=g\oplus g,
\]
where $g$ is the round metric of ${\mathbb S}^2$. It is well known that $g_+$ is Einstein with scalar curvature $R=4$. 

\noindent Define the almost paracomplex structure $P$ on ${\mathbb S}^2\times {\mathbb S}^2$ by:
\[
P(u,v)=(u,-v),
\]
where $(u,v)\in T({\mathbb S}^2\times {\mathbb S}^2)$. Then, $P$ is $G^+$-parallel and isometric. For $t\in (-1,1)$, consider the homogeneous hypersurfaces:
\[
\Sigma_t=\{(x,y)\in {\mathbb S}^2\times {\mathbb S}^2\subset {\mathbb R}^3\times {\mathbb R}^3)\,|\, \left<x,y\right>=t\}.
\]
In fact, $\Sigma_t$ is a tube of radius $\cos^{-1}(t/\sqrt{2})$ over the diagonal surface $\Delta=\{(x,x)\in  {\mathbb S}^2\times {\mathbb S}^2\}$. It was shown in \cite{urbano}, that $\Sigma_t$ is null for every $t$ with respect to the neutral metric 
\[
g_-=g_+(P.,.)=g\oplus (-g)
\]
and the principal curvatures are
\[
\lambda_1=\textstyle{\frac{1}{\sqrt{2}}}\sqrt{\frac{1+t}{1-t}},\qquad \lambda_2=-\frac{1}{\sqrt{2}}\sqrt{\frac{1-t}{1+t}},\qquad \lambda_3=0.
\]
Thus, $\Sigma_t$ is a CMC null hypersurface for any $t\in (-1,1)$ and is minimal only when $t=0$ as the mean curvature $H$ is
\[
H=\textstyle{\frac{1}{3\sqrt{2}}}\left(\sqrt{\frac{1+t}{1-t}}-\sqrt{\frac{1-t}{1+t}}\right).
\]
\end{exmp}

\noindent Similarly, we have the following example.

\begin{exmp}
Consider the Cartesian product of the 2-spheres ${\mathbb H}^2\times {\mathbb H}^2$ endowed with the product metric 
\[
g_+=g\oplus g,
\]
where $g$ is the standard hyperbolic metric of ${\mathbb H}^2$. It is not hard for one to see that $g_+$ is Einstein with scalar curvature $R=-4$. As before, the almost paracomplex structure $P$ on ${\mathbb H}^2\times {\mathbb H}^2$ is given by:
\[
P(u,v)=(u,-v),
\]
where $(u,v)\in T({\mathbb H}^2\times {\mathbb H}^2)$. Again, $P$ is $g_+$-parallel and isometric and for $t\in (-1,1)$, consider the homogeneous hypersurfaces:
\[
\Sigma_t=\{(x,y)\in {\mathbb H}^2\times {\mathbb H}^2\subset {\mathbb R}^3\times {\mathbb R}^3)\,|\, \left<x,y\right>=t\}.
\]
In fact, $\Sigma_t$ is a tube of radius $\cosh^{-1}(t/\sqrt{2})$ over the diagonal surface $\Delta=\{(x,x)\in  {\mathbb H}^2\times {\mathbb H}^2\}$. It was shown in \cite{GMY} that $\Sigma_t$ is null for every $t$ with respect to the neutral metric 
\[
g_-=g_+(P.,.)=g\oplus (-g)
\]
and the principal curvatures are
\[
\lambda_1=\textstyle{\frac{1}{\sqrt{2}}}\sqrt{\frac{1+t}{1-t}},\qquad \lambda_2=\frac{1}{\sqrt{2}}\sqrt{\frac{1-t}{1+t}},\qquad \lambda_3=0.
\]
Thus, $\Sigma_t$ is a CMC, non-minimal null hypersurface for any $t\in (-1,1)$ with mean curvature:
\[
H=\textstyle{\frac{1}{3\sqrt{2}}}\left(\sqrt{\frac{1+t}{1-t}}+\sqrt{\frac{1-t}{1+t}}\right).
\]
\end{exmp}

\vspace{0.2in}

\subsection{Main results}

Consider the principal orthonormal frame $(e_1,e_2,e_3=PN)$ of the null hypersurface $\Sigma$ so that
\[
Ae_i=\lambda_i e_i.
\]
 It is easily shown  that there is an angle $\theta\in [0,2\pi)$ such that
\[
Pe_1=\cos\theta e_1+\sin\theta e_2\qquad Pe_2=\sin\theta e_1-\cos\theta e_2.
\]
We call the angle $\theta$ \emph{the principal angle} of the null hypersurface $\Sigma$.

\noindent We now have the following result for totally geodesic null hypersurfaces:

\begin{Thm}
Every totally geodesic null hypersurface is scalar flat. If $M$ admits a totally geodesic null hypersurface then $(M,g_+)$ is Ricci-flat. 
\end{Thm}

\begin{proof}
Let $\{e_1,e_2.e_3\}$ be an orthonormal frame of $\Sigma$ such that 
\[
Ae_i=\lambda_i e_i,
\]
where $e_3=PN$ and therefore, $\lambda_3=0$. 
The almost paracomplex structure $P$ is :
\[
P=\begin{pmatrix}\cos\theta & \sin\theta & 0 & 0\\ \sin\theta & -\cos\theta & 0 & 0\\ 0 & 0 & 0 & 1\\ 0 & 0 &1 & 0  \end{pmatrix}.
\]
with respect to the orthonormal frame $(e_1,e_2,e_3,N)$.

\noindent Let $\overline\nabla,\nabla$ be the Levi-Civita connections for the metrics $g$ and the induced metric of $g$ on $\Sigma$, respectively. For $i,j=1,2,3$ we have
\[
\overline\nabla_{e_i}e_j=\nabla_{e_i}e_j+\lambda_i\delta_{ij}N,
\]
and if we let $\omega_{ij}^k=g(\nabla_{e_i}e_j,e_k)$ then
\[
\omega_{ij}^k=-\omega_{ik}^j.
\]
Defining
\begin{equation}\label{e:defofkmunu}
k=\omega_{11}^2,\qquad \mu=\omega_{21}^2,\qquad \nu=\omega_{31}^2.
\end{equation}
A brief calculation gives
\[
g(R(e_2,e_1)e_1,e_2)=-e_1(\mu)+e_2(k)+\lambda_1\lambda_2-k^2-\mu^2+\nu (\lambda_1-\lambda_2)\sin\theta.
\]
\[
g(R(e_3,e_1)e_1,e_3)=-\lambda_1\nu\sin\theta-\lambda_2\nu\sin\theta+e_3(\lambda_1\cos\theta)-\lambda_1^2\cos^2\theta-\lambda_1\lambda_2\sin^2\theta.
\]
\[
g(R(e_3,e_2)e_2,e_3)=\lambda_1\nu\sin\theta+\lambda_2\nu\sin\theta-e_3(\lambda_2\cos\theta)-\lambda_2^2\cos^2\theta-\lambda_1\lambda_2\sin^2\theta.
\]
Therefore, we deduce
\[
\mbox{Ric}(e_1,e_1)=-e_1(\mu)+e_2(k)+e_3(\lambda_1\cos\theta)+\lambda_1\lambda_2\cos^2\theta-\lambda^2_1\cos^2\theta-k^2-\mu^2-2\nu\lambda_2\sin\theta.
\]
\[
\mbox{Ric}(e_2,e_2)=-e_1(\mu)+e_2(k)-e_3(\lambda_2\cos\theta)+\lambda_1\lambda_2\cos^2\theta-\lambda^2_2\cos^2\theta-k^2-\mu^2+2\nu\lambda_1\sin\theta.
\]
\[
\mbox{Ric}(e_3,e_3)=-e_3[(\lambda_1-\lambda_2)\cos\theta]-2\lambda_1\lambda_2\sin^2\theta-(\lambda^2_1+\lambda^2_2)\cos^2\theta.
\]
The scalar curvature $R$ of $\Sigma$ is
\begin{eqnarray}
R&=&-2e_1(\mu)+2e_2(k)+2\lambda_1\lambda_2\cos 2\theta-2(\lambda^2_1+\lambda^2_2)\cos^2\theta\label{e:scalar curvature1} \\
&&\qquad\qquad\qquad\qquad\qquad\qquad -2k^2-2\mu^2+2\nu(\lambda_1-\lambda_2)\sin\theta.\nonumber 
\end{eqnarray}
Using the fact that $P$ is parallel, namely 
\[
P\overline\nabla_{e_i}e_j=\overline\nabla_{e_i}Pe_j,
\]
we have,
\begin{equation}\label{e:connection}
\omega_{12}^3=\lambda_1\sin\theta,\qquad \omega_{11}^3=\lambda_1\cos\theta, \qquad \omega_{12}^1=e_1(\theta/2),
\end{equation}
\[
\omega_{21}^3=\lambda_2\sin\theta,\qquad \omega_{22}^3=-\lambda_2\cos\theta, \qquad \omega_{22}^1=e_2(\theta/2),
\]
\[
\omega_{13}^1=-\lambda_1\cos\theta\qquad \omega_{13}^2=-\lambda_1\sin\theta\qquad \omega_{23}^1=-\lambda_2\sin\theta,
\]
\[
\omega_{31}^2=-e_3(\theta/2),\qquad \omega_{31}^3=\omega_{32}^3=0,
\]
and thus,
\[
\nabla_{e_1}e_1=-e_1(\theta/2)e_2+\lambda_1\cos\theta e_3
\qquad
\nabla_{e_1}e_2=e_1(\theta/2)e_1+\lambda_1\sin\theta e_3
\]
\[
\nabla_{e_1}e_3=-\lambda_1\cos\theta e_1-\lambda_1\sin\theta e_2
\qquad
\nabla_{e_2}e_1=-e_2(\theta/2)e_2+\lambda_2\sin\theta e_3
\]
\[
\nabla_{e_2}e_2=e_2(\theta/2)e_1-\lambda_2\cos\theta e_3
\qquad
\nabla_{e_2}e_3=-\lambda_2\sin\theta e_1+\lambda_2\cos\theta e_2
\]
\[
\nabla_{e_3}e_1=-e_3(\theta/2)e_2
\qquad
\nabla_{e_3}e_2=e_3(\theta/2)e_1
\qquad
\nabla_{e_3}e_3=0.
\]
The relations (\ref{e:defofkmunu}) and (\ref{e:connection}), yield
\[
\mu=-e_2(\theta/2),\qquad k=-e_1(\theta/2),
\]
and therefore,
\[
-e_1(\mu)+e_2(k)=[e_1,e_2](\theta/2).
\]
On the other hand,
\begin{eqnarray}
[e_1,e_2]&=&e_1(\theta/2)e_1+\lambda_1\sin\theta e_3-(-e_2(\theta/2)e_2+\lambda_2\sin\theta e_3)\nonumber\\
&=&e_1(\theta/2)e_1+e_2(\theta/2)e_2+(\lambda_1-\lambda_2)\sin\theta\,e_3.\nonumber
\end{eqnarray}
Thus,
\begin{eqnarray}
-e_1(\mu)+e_2(k)&=&[e_1,e_2](\theta/2)\nonumber\\
&=&e_1(\theta/2)e_1(\theta/2)+e_2(\theta/2)e_2(\theta/2)+(\lambda_1-\lambda_2)\sin\theta\,e_3(\theta/2)\nonumber\\
&=&k^2+\mu^2-\nu (\lambda_1-\lambda_2)\sin\theta\nonumber
\end{eqnarray}
The scalar curvature given in (\ref{e:scalar curvature1}) now becomes
\begin{equation}\label{e:scalar curvature2}
R=2\lambda_1\lambda_2\cos 2\theta-2(\lambda^2_1+\lambda^2_2)\cos^2\theta.
\end{equation}
Assuming that $\Sigma$ is totally geodesic we can see easily that $R=0$. In this case, the Gauss equation implies also that $(M,g)$ is scalar flat since
\[
\frac{\bar R}{2}=R-9H^2+|\sigma|^2=0.
\]
The Ricci flatness of $(M,g)$ follows from the fact $g$ is Einstein.   
\end{proof}

\vspace{0.1in}

\noindent If $\Sigma$ is a null hypersurface, the principal curvature corresponding to the principal direction $PN$ will be called \emph{trivial}. The following Theorem explores null hypersurfaces where the non-trivial eigenvalues are equal.

\begin{Thm}\label{t:main} 
Suppose $(M,g)$ has nonnegative scalar curvature and $\Sigma$ is a null hypersurface with equal non-trivial principal curvatures. Then, $g$ is Ricci-flat and $\Sigma$ is totally geodesic.
\end{Thm}
\begin{proof}
Using the scalar curvature $R$ in (\ref{e:scalar curvature2}), the Gauss equation for $\Sigma$ becomes
\[
\frac{\bar R}{2}+(\lambda_1+\lambda_2)^2=-(\lambda_1-\lambda_2)^2\cos2\theta.
\]
Since $\lambda_1=\lambda_2$, we have 
\[
\frac{\bar R}{2}+(\lambda_1+\lambda_2)^2=0,\]
impliying $\bar R=0$ and $\lambda_1+\lambda_2=0$. This means that $\lambda_1=\lambda_2=0$ and thus, $\Sigma$ is totally null.
\end{proof}

\vspace{0.1in}

\noindent We now have the following theorem about CMC null hypersurfaces:

\begin{Thm}\label{t:main} 
Let $\Sigma$ be a CMC, non-minimal null hypersurface in $(M,g)$. Then, all principal curvatures and the scalar curvature of $\Sigma$ are constant. Furthermore, the scalar curvature of $g$ is given by
\begin{equation}\label{e:cmc}
\bar R=-8\lambda_1\lambda_2,
\end{equation}
where $\lambda_1,\lambda_2$, denote the non-trivial principal curvatures of $\Sigma$.
\end{Thm}
\begin{proof}
We recall the principal orthonormal frame $\{e_1,e_2.e_3=PN\}$ of the null hypersurface $\Sigma$. The Laplacian of the function $C$ with respect to the induced metric is
\[
\Delta C=-6\,g(X,\nabla H)-2C|\sigma|^2+2\mbox{Tr}(P^TA^2).
\]
Since $C=0$ and $\nabla H=0$, we have
\[
\mbox{Tr}(P^TA^2)=0,
\]
which ensures
\[
\sum_{i=1}^3g(PA^2e_i,e_i)=0.
\]
It follows
\[
\sum_{i=1}^2\lambda^2_ig(Pe_i,e_i)=0,
\]
and therefore,
\[
(\lambda^2_1-\lambda^2_2)\cos\theta=0.
\]
Note that $\Sigma$ is non-minimal and therefore, $\lambda_1+\lambda_2\neq 0$. 

\noindent If $\lambda_1=\lambda_2$, we have that $H=\textstyle{\frac{2}{3}}\lambda_1$ is constant and considering the scalar curvature in (\ref{e:scalar curvature2}), we find
\begin{eqnarray}
\textstyle{\frac{1}{2}}R&=&\lambda^2_1\cos 2\theta-2\lambda^2_1\cos^2\theta\nonumber\\
&=&-\lambda^2_1.\nonumber
\end{eqnarray}
Using the Gauss equation (\ref{e:scalar curvature}), we obtain
\begin{eqnarray}
-2\lambda^2_1&=& R\nonumber \\
&=&\textstyle{\frac{1}{2}}\bar R+9H^2-|\sigma|^2\nonumber \\
&=&\textstyle{\frac{1}{2}}\bar R+(2\lambda_1)^2-2\lambda^2_1,\nonumber 
\end{eqnarray}
which implies that $\bar R=-8\lambda^2_1$. 

\noindent If $\cos\theta=0$, then either $\theta=\pi/2$ or $\theta=3\pi/2$.
The scalar curvature of $\Sigma$ given (\ref{e:scalar curvature2}) becomes
\[
R=-2\lambda_1\lambda_2.
\]
On the other hand, the scalar curvature in (\ref{e:scalar curvature}) yields
\begin{eqnarray}
-2\lambda_1\lambda_2&=& R\nonumber \\
&=&\textstyle{\frac{1}{2}}\bar R+9H^2-|\sigma|^2\nonumber \\
&=&\textstyle{\frac{1}{2}}\bar R+(\lambda_1+\lambda_2)^2-\lambda^2_1-\lambda^2_2,\nonumber 
\end{eqnarray}
and therefore, $\bar R=-8\lambda_1\lambda_2$. Note that $\bar R$ is constant and as such $\lambda_1\lambda_2$ is constant. However $\lambda_1+\lambda_2$ is also constant and thus both $\lambda_1$ and $\lambda_2$ are constant.

\noindent All principal curvatures are constant and therefore the Gauss equation, given in (\ref{e:scalar curvature}), tells us that the scalar curvature $R$ must also be constant. 

\end{proof}

\vspace{0.2in}

\noindent Theorem \ref{t:main} can no longer be extended to minimal null hypersurfaces since the relation (\ref{e:cmc}) does not necessarily hold. To see this, consider the minimal, null hypersurfaces $M_{a,b}\subset {\mathbb S}^2\times {\mathbb S}^2$, for $a,b\in {\mathbb S}^2\subset {\mathbb R}^3$:
\[
M_{a,b}=\{(x,y)\in  {\mathbb S}^2\times  {\mathbb S}^2\,|\, \left<x,a\right>+\left<y,b\right>=0\}.
\]
In \cite{urbano} F. Urbano showed that the principal curvatures are non constant and in particular, if $(x,y)\in M_{a,b}$ then: 
\[
\textstyle{\lambda_1(x,y)}=\textstyle{\frac{\left<x,a\right>}{\sqrt{2(1-\left<x,a\right>^2)}}},\qquad \lambda_2(x,y)=-\frac{\left<x,a\right>}{\sqrt{2(1-\left<x,a\right>^2)}},\qquad \lambda_3(x,y)=0.
\]
As such 
\[
-8\lambda_1\lambda_2=\textstyle{\frac{4\left<x,a\right>^2}{1-\left<x,a\right>^2}}\neq 4=\bar R.
\]

\end{document}